\documentclass[11pt, reqno]{amsart}
\usepackage{amsfonts}
\usepackage{amsmath}
\usepackage{amssymb, color}
\usepackage{amscd}
\usepackage[mathscr]{eucal}
\usepackage{url}
\usepackage{enumerate}

\usepackage{tikz}
\usetikzlibrary{decorations.markings, arrows.meta}
\tikzset{
mid arrow/.style={postaction={decorate, decoration={
markings,
mark=at position 1.0 with {\arrow{Straight Barb}}
}}},
}
\allowdisplaybreaks[1]

\theoremstyle{plain}

\newtheorem{theorem}{Theorem}[section]

\newtheorem{prop}[theorem]{Proposition}
\newtheorem{cor}[theorem]{Corollary}
\theoremstyle{definition}

\numberwithin{equation}{section}

\newcommand{\N}{\textup{N}}
\newcommand{\IN}{\textup{IN}}
\newcommand{\ON}{\textup{ON}}

\makeatletter
\def\imod#1{\allowbreak\mkern5mu({\operator@font mod}\,\,#1)}
\makeatother

\makeatletter
\@namedef{subjclassname@2020}{%
  \textup{2020} Mathematics Subject Classification}
\makeatother

\begin{document}

\title[Mathon-type digraphs]{A Mathon-type construction for digraphs and improved lower bounds for Ramsey numbers}

\author{Dermot M\lowercase{c}Carthy, Chris Monico}

\address{Dermot M\lowercase{c}Carthy, Department of Mathematics \& Statistics, Texas Tech University, Lubbock, TX 79410-1042, USA}
\email{dermot.mccarthy@ttu.edu}
\urladdr{https://www.math.ttu.edu/~mccarthy/}

\address{Chris Monico, Department of Mathematics \& Statistics, Texas Tech University, Lubbock, TX 79410-1042, USA}
\email{C.Monico@ttu.edu}
\urladdr{https://www.math.ttu.edu/~cmonico/}


\begin{abstract}
We construct an edge-colored digraph analogous to Mathon's construction for undirected graphs.
We show that this graph is connected to the $k$-th power Paley digraphs and we use this connection to produce improved lower bounds for multicolor directed Ramsey numbers. 
\end{abstract}

\maketitle


\section{Introduction}\label{sec_Intro}
In \cite{Ma}, Mathon leveraged properties of generalized Paley graphs to improve lower bounds on diagonal multicolor (undirected) Ramsey numbers. He did this by constructing a multicolored graph which contained monochromatic induced subgraphs isomorphic to the generalized Paley graph. Among his results were $R(7,7) \geq 205$, $R(9,9) \geq 565$, $R(10,10) \geq 798$ and $R_3(4) \geq 128$, which are still the best known lower bounds today \cite{R}. Independently, Shearer \cite{Sh} produced the same results in the two-color case using an equivalent construction. More recently, Xu and Radziszowski \cite{XR} made incremental improvements to Mathon's construction and showed that $R_3(7) \geq 3214$ (increased from Mathon's 3211), which is the current best known lower bound.  

In this paper, we adapt Mathon's construction to digraphs and leverage properties of $k$-th power Paley digraphs to produce improved lower bounds for diagonal multicolor directed Ramsey numbers. For the remainder of this paper all Ramsey numbers will be directed, and will be denoted $R_t(m)$. As such, $R_t(m)$ is the least positive integer $n$ such that any tournament with $n$ vertices, whose edges have been colored in $t$ colors, contains a monochromatic transitive subtournament of order $m$. When $t=1$ we recover the usual directed Ramsey number $R(m)$, so we drop the subscript in this case. Recall, a tournament is transitive if, whenever $a \to b$ and $b \to c$, then $a \to c$. 
Our main results which improve on the previously best known lower bounds can be summarized as follows.
\begin{theorem}\label{thm_k=2}
$R(8) \geq 57, R(11) \geq 169, R(12) \geq 217, R(14) \geq 401, R(15) \geq 545, R(16) \geq 737, R(17) \geq 889, R(18)\geq 1241, R(19) \geq 1321$ and  $R(20) \geq 1945.$
\end{theorem}

\begin{theorem}\label{thm_k>2}
For $t \geq 4$,
$$R_t(3) \geq 169 \cdot 3^{t-4}+1.$$
For $t \geq 2$,
$$R_t(6) \geq 829 \cdot 27^{t-2} +1  \qquad \text{and} \qquad
R_t(8) \geq 3320 \cdot 56^{t-2} +1.$$
\end{theorem}


\section{Preliminaries and Notation}\label{sec_Notation}
For a graph $G$, we denote its vertex set by $V(G)$, so the order of $G$ is $\#V(G)$. 
For a vertex $v$ of a digraph $G$, we will denote the set of vertices which are out-neighbors of $v$ by $\ON(v)$ and the set of in-neighbors by $\IN(v)$. If the edges of $G$ are colored, we will denote the set of out-neighbors (resp.~in-neighbors) of $v$ connected via an edge of color $i$ by $\ON_i(v)$ (resp. $\IN_i(v)$). We define the set of neighbors of $v$ as $\N(v) := \ON(v) \cap \IN(v)$ and the set of color $i$ neighbors as $\N_i(v) := \ON_i(v) \cap \IN_i(v)$. We will refer to any collection of vertices in $G$, which are pairwise connected via two edges oriented in opposite directions, as a clique. Further, if all those edges are of color $i$, we will refer to it as a color $i$ clique.

We note that a tournament of order $m$ is transitive if and only if the set of out-degrees of its vertices is $\{0,1, \ldots ,m-1\}$ \cite[Ch.~7]{Mo}. Thus, we can represent a transitive subtournament of order $m$ by the $m$-tuple of its vertices $(a_1, a_2, \ldots, a_m)$, listed in order such that the out-degree of vertex $a_i$ is $m-i$, i.e. the corresponding $m$-tuple of out-degrees is $(m-1, m-2, \ldots, 1, 0)$. 
We let $\mathcal{K}_m(G)$ denote the number of transitive subtournaments of order $m$ contained in a digraph $G$.


\section{Mathon-Type Construction for Digraphs}\label{sec_Con}
Let $k \geq 2$ be an even integer. Let $q$ be a prime power such that $q \equiv k+1 \imod {2k}$.
This condition ensures that $-1$ is not a $k$-th power in $\mathbb{F}_q$, the finite field with $q$ elements, but is a $\frac{k}{2}$-th power.
Let $S_k$ be the subgroup of the multiplicative group $\mathbb{F}_q^{\ast}$ of order $\frac{q-1}{k}$ containing the $k$-th power residues, i.e., if $\omega$ is a primitive element of $\mathbb{F}_q$, then $S_k = \langle \omega^k \rangle$. 
We define $S_{k,0}:=\{0\}$ and $S_{k,i} := \omega^{i-1} S_k$, for $1 \leq i \leq \tfrac{k}{2}$, so that $S_{k,1}=S_k$.
We note that $-S_{k,i}= \omega^{\frac{k}{2}} S_{k,i}$ (as $-1=\omega^{\frac{q-1}{2}}$ and $\frac{q-1}{2} \equiv \frac{k}{2} \imod{k}$), yielding the disjoint union $$\mathbb{F}_q = S_{k,0} \, \cup \, \bigcup_{i=1}^{k/2} S_{k,i} \, \cup \, \bigcup_{i=1}^{k/2} -S_{k,i}.$$

Let $X:=(\mathbb{F}_q \times \mathbb{F}_q) \setminus \{(0,0)\}$. We define an equivalence relation $\sim$ on $X$ where $(a,b) \sim (c,d)$ if $(c,d)=(ag,bg)$ for some $g \in S_k$. We denote the equivalence class of $(a,b)$ by $[a,b]$. There are $n:=k(q+1)$ such equivalence classes, each containing $|S_k| =\frac{q-1}{k}$ elements.
Let $M_k(q)$ be the edge-colored digraph of order $n$, with vertex set $X/\sim$, where $[a,b] \to [c,d]$ is an edge in color $i$, $0 \leq i \leq \frac{k}{2}$, if and only if $bc-ad \in S_{k,i}$.
We note that this is well-defined as $g S_{k,i}= S_{k,i}$ for all $g \in S_k$. We also note that any pair of vertices of $M_k(q)$ will either be connected by a single oriented edge in color $i$, for some $i \leq i \leq \frac{k}{2}$, or, connected by two edges of color $0$ oriented in opposite directions. 
For ease of illustration in what follows, we will represent the former case by $v_1 \stackrel{i}{\to} v_2$ and the latter case by $v_1 \stackrel{0}{\longleftrightarrow} v_2$.

\begin{prop}\label{prop_VertTrans}
$M_k(q)$ is vertex transitive.
\end{prop}

\begin{proof}
For $s \in \mathbb{F}_q$, define the maps $\rho_s$ and $\sigma_s$ on $X/\sim$ by
\begin{align*}
\rho_s&: [a,b] \to [a, b+as] \\
\sigma_s&: [a,b] \to [a+bs,b].
\end{align*}
It is easy to show that both $\rho_s$ and $\sigma_s$ are well-defined automorphisms of $M_k(q)$.
Let $[a,b]$ and $[c,d]$ be distinct vertices of $M_k(q)$. 
Assume first that $b,c \neq 0$ and let $s_1, s_2 \in \mathbb{F}_q$ satisfy 
$a+b s_1 = c$
and
$b+c s_2 = d$.
Then $\rho_{s_2} ( \sigma_{s_1} [a,b]) =[c,d]$.
If $b=0$ then $a \neq 0$, and we can first apply $\rho_1[a,0]=[a,a]$ and then proceed as before.
If $c=0$ then $d \neq 0$, and we can proceed as before to get to $[d,d]$. Then we apply $\sigma_{-1}[d,d] = [0,d]$.
\end{proof}

\begin{prop}\label{prop_Gamma}
For $0 \leq i \leq \frac{k}{2}$, let $\Gamma_i$ be the subgraph of $M_k(q)$, with vertex set $X/\sim$, induced by the color $i$ edges of $M_k(q)$.
\begin{enumerate}[ (1) ]
\item
$\Gamma_0$ is the disjoint union of $q+1$ color $0$ cliques of order $k$.
\item
$\Gamma_1, \Gamma_2, \ldots, \Gamma_{\frac{k}{2}}$ are pairwise isomorphic.
\end{enumerate} 
\end{prop}

\begin{proof}
(1) The neighbors of [0,1] in $\Gamma_0$ are $\N_0([0,1]) = \{[0, \omega^j] \mid j=1,2,\ldots, k-1\}$.
All elements of $\N_0([0,1])$ are neighbors of each other in $\Gamma_0$ and, thus, $[0,1]$ and its neighbors form a clique of order $k$.
As $M_k(q)$ is vertex transitive, every vertex belongs to such a clique.
And, as the elements of $\N_0([0,1])$ are not neighbors of any other vertices in $\Gamma_0$, all such cliques are disjoint.
Therefore, there must be $\frac{n}{k}=q+1$ of them.
(2) $\Gamma_i$ is isomorphic to $\Gamma_{i+1}$, for all $1 \leq i \leq \frac{k}{2}-1$, via the map $[a,b] \to [wa,b]$.
\end{proof}

\begin{prop}\label{prop_Nxv}
Let $v \in V(M_k(q))$.
Let $x \in \N_0(v)$. Then for any $i \in \{1,2, \ldots, \frac{k}{2} \}$,
$$\ON_i(x) \cap \ON_i(v) = \IN_i(x) \cap \IN_i(v) =  \emptyset.$$
\end{prop}

\begin{proof}
As $M_k(q)$ is vertex transitive, it suffices to prove for $v=[0,1]$.
Then, let $x \in \N_0([0,1])$, i.e., $x=[0, w^j]$ for some $j=1,2,\ldots, k-1$.
Now
\begin{equation*}
[0, \omega^j] \stackrel{i}{\to} [c,d]
\Longleftrightarrow
\omega^j c \in S_{k,i}
\Longleftrightarrow
c \in \{ \omega^{kl+i-j-1} \mid l=0,1,\ldots, \tfrac{q-1}{k}-1\},
\end{equation*}
and so
$$\ON_i(x) = \ON_i([0, \omega^j]) = \{ [ \omega^{i-j-1 \imod{k}}, d] \mid d \in \mathbb{F}_q \}.$$
Also,
$$\ON_i(v) = \ON_i([0, 1]) = \{ [ \omega^{i-1}, d] \mid d \in \mathbb{F}_q \}.$$
As $j \not\equiv 0 \imod{k}$, we get that $\ON_i(x) \cap \ON_i(v) = \emptyset$.
Similar arguments produce
$$\IN_i(x) = \IN_i([0, \omega^j]) = \{ [ \omega^{i-j-1+\frac{k}{2} \imod{k}}, b] \mid b \in \mathbb{F}_q \}$$
and
$$\IN_i(v) = \IN_i([0, 1]) = \{ [ \omega^{i-1+\frac{k}{2}}, b] \mid b \in \mathbb{F}_q \}.$$
So, $\IN_i(x) \cap \IN_i(v) = \emptyset$.
\end{proof}


\section{Relation to the $k$-th power Paley digraphs}\label{sec_Paley}
Recall from Section \ref{sec_Con}, $k \geq 2$ is an even integer and $q$ is a prime power such that $q \equiv k+1 \imod {2k}$. $S_k$ is the subgroup of $\mathbb{F}_q^{\ast}$ containing the $k$-th power residues, i.e., if $\omega$ is a primitive element of $\mathbb{F}_q$, then $S_k = \langle \omega^k \rangle$, and $S_{k,i} := \omega^{i-1} S_k$, for $1 \leq i \leq \tfrac{k}{2}$.

We now recall some definitions and properties from \cite{MS} concerning Paley digraphs. We define the \emph{k-th power Paley digraph} of order $q$, $G_k(q)$, as the graph with vertex set $\mathbb{F}_q$ where $a \to b$ is an edge if and only if $b-a \in S_k$. We note that $-1 \notin S_k$ so $G_k(q)$ is a well-defined oriented graph.
For each $1 \leq i \leq \frac{k}{2}$, we define the related directed graph $G_{k,i}(q)$ with vertex set $\mathbb{F}_q$ where $a \to b$ is an edge if and only if $b-a \in S_{k,i}$. Each $G_{k,i}(q)$ is isomorphic to $G_{k,1}(q)=G_k(q)$, the $k$-th power Paley digraph, via the map $f_i:V(G_k(q)) \to V(G_{k,i}(q))$ given by $f_i(a)=\omega^{i-1} a$. Now consider the \emph{multicolor $k$-th power Paley tournament} $P_k(q)$ whose vertex set is $\mathbb{F}_q$ and whose edges are colored in $\frac{k}{2}$ colors according to $a \to b$ has color $i$ if $b-a \in S_{k,i}$. Note that the induced subgraph of color $i$ of $P_k(q)$ is $G_{k,i}(q)$. Thus, $P_k(q)$ contains a monochromatic transitive subtournament of order $m$ if and only if $G_k(q)$ contains a transitive subtournament of order $m$.

\begin{prop}\label{prop_ONPaley}
Let $i \in \{1,2, \ldots, \frac{k}{2} \}$. 
Let $v \in V(M_k(q))$.
Then the induced subgraph of $M_k(q)$ with vertex set $\ON_i(v)$ is isomorphic to $P_k(q)$.
\end{prop}

\begin{proof}
As $M_k(q)$ is vertex transitive, it suffices to prove for $v=[0,1]$.
Let $H$ denote the induced subgraph of $M_k(q)$ with vertex set $\ON_i([0,1])$.
In the proof of Proposition \ref{prop_Nxv} we saw that $\ON_i([0, 1]) = \{ [ \omega^{i-1}, d] \mid d \in \mathbb{F}_q \}.$
So $\#V(H) = | \ON_i([0, 1]) |= q = \#V(P_k(q))$.
Now consider the bijective map $\phi: V(H) \to V(P_k(q))$ given by $\phi([ \omega^{i-1}, d]) = -\omega^{i-1}d$.
It remans to show that $\phi$ is color-preserving. 
Let $[ \omega^{i-1}, d_1] \in V(H)$ and let $[ \omega^{i-1}, d_2] \in \ON_s([ \omega^{i-1}, d_1])]$  for some $s \in \{1,2, \ldots, \frac{k}{2} \}$ (note that $s \neq 0$ otherwise $d_1=d_2$).
Now,
\begin{align*}
[ \omega^{i-1}, d_1] \stackrel{s}{\to} [ \omega^{i-1}, d_2]
&\Longleftrightarrow
d_1 \omega^{i-1} - \omega^{i-1} d_2 \in S_{k,s}\\
&\Longleftrightarrow
\phi([ \omega^{i-1}, d_2]) - \phi([ \omega^{i-1}, d_1]) \in S_{k,s}\\
&\Longleftrightarrow
\phi([ \omega^{i-1}, d_1]) \stackrel{s}{\to} \phi([ \omega^{i-1}, d_2]),
\end{align*}
as required.
\end{proof}

Recall that any pair of vertices of $M_k(q)$ will either be connected by a single oriented edge in color $i$, for some $1 \leq i \leq \frac{k}{2}$, or, connected by two edges of color $0$ oriented in opposite directions. We now replace all these pairs of color $0$ edges with a single oriented edge of color $1\leq i \leq \frac{k}{2}$, where the new color and orientation are randomly assigned. 
We call this altered graph $M_k^{\ast}(q)$, which is a tournament whose edges are colored in $\frac{k}{2}$ colors.

\begin{theorem}\label{thm_Main}
Let $k \geq 2$ be an even integer and $q$ be a prime power such that $q \equiv k+1 \imod {2k}$. 
Let $m \geq k$. 
If $P_k(q)$ contains no monochromatic transitive subtournament of order $m$, then $M_k^{\ast}(q)$ contains no monochromatic transitive subtournament of order $m+2$.
\end{theorem}

\begin{proof}
Let $T_l^{\ast}$ be a monochromatic, in color $i$, $1 \leq i \leq \frac{k}{2}$, transitive subtournament of $M_k^{\ast}(q)$ of order $l$. 
We represent $T_l^{\ast}$ by the $l$-tuple of its vertices $(a_1, a_2, \ldots, a_l)$ with the corresponding $l$-tuple of out-degrees $(l-1, l-2, \ldots, 1, 0)$.
Let $T_l$ be the corresponding subgraph of $M_k(q)$ before the color $0$ edges were reassigned, i.e., $T_l$ also has vertices $a_1, a_2, \ldots, a_l$ but some vertices may be connected by two edges of color $0$ oriented in opposite directions. 

Assume $a_1 \stackrel{0}{\longleftrightarrow} a_2$ in $M_k(q)$. 
If $l \geq 2$, consider $a_t$ for $3 \leq t \leq l$.
Then there are four possibilities for the triangle $(a_1, a_2, a_t)$ in $M_k(q)$:\\

\begin{tikzpicture}
\node[inner sep=2pt, circle] (A) at (-2,0) {$a_1$};
\node[inner sep=2pt, circle] (B) at (0,0) {$a_2$}; 
\node[inner sep=2pt, circle] (C) at (-1,-1.5) {$a_t$}; 

\draw[thick, mid arrow] (A.east) to (B.west);
\draw[thick, mid arrow] (B.west) to (A.east);
\draw[thick, mid arrow] (A.south east) to (C.north west);
\draw[thick, mid arrow] (B.south west) to (C.north east);

\node at (-1,0.25) {\small{$0$}};
\node at (-1.75,-0.75) {\small{$i$}};
\node at (-0.25,-0.75) {\small{$i$}};
\node at (-1.0,-2.25) {\small{(1)}};


\node[inner sep=2pt, circle] (A) at (2,0) {$a_1$};
\node[inner sep=2pt, circle] (B) at (4,0) {$a_2$}; 
\node[inner sep=2pt, circle] (C) at (3,-1.5) {$a_t$}; 

\draw[thick, mid arrow] (A.east) to (B.west);
\draw[thick, mid arrow] (B.west) to (A.east);
\draw[thick, mid arrow] (A.south east) to (C.north west);
\draw[thick, mid arrow] (B.south west) to (C.north east);
\draw[thick, mid arrow] (C.north east) to (B.south west);

\node at (3,0.25) {\small{$0$}};
\node at (2.25,-0.75) {\small{$i$}};
\node at (3.75,-0.75) {\small{$0$}};
\node at (3,-2.25) {\small{(2)}};


\node[inner sep=2pt, circle] (A) at (6,0) {$a_1$};
\node[inner sep=2pt, circle] (B) at (8,0) {$a_2$}; 
\node[inner sep=2pt, circle] (C) at (7,-1.5) {$a_t$}; 

\draw[thick, mid arrow] (A.east) to (B.west);
\draw[thick, mid arrow] (B.west) to (A.east);
\draw[thick, mid arrow] (A.south east) to (C.north west);
\draw[thick, mid arrow] (C.north west) to (A.south east);
\draw[thick, mid arrow] (B.south west) to (C.north east);

\node at (7,0.25) {\small{$0$}};
\node at (6.25,-0.75) {\small{$0$}};
\node at (7.75,-0.75) {\small{$i$}};
\node at (7,-2.25) {\small{(3)}};


\node[inner sep=2pt, circle] (A) at (10,0) {$a_1$};
\node[inner sep=2pt, circle] (B) at (12,0) {$a_2$}; 
\node[inner sep=2pt, circle] (C) at (11,-1.5) {$a_t$}; 

\draw[thick, mid arrow] (A.east) to (B.west);
\draw[thick, mid arrow] (B.west) to (A.east);
\draw[thick, mid arrow] (A.south east) to (C.north west);
\draw[thick, mid arrow] (C.north east) to (B.south west);
\draw[thick, mid arrow] (C.north west) to (A.south east);
\draw[thick, mid arrow] (B.south west) to (C.north east);

\node at (11,0.25) {\small{$0$}};
\node at (10.25,-0.75) {\small{$0$}};
\node at (11.75,-0.75) {\small{$0$}};
\node at (11,-2.25) {\small{(4)}};
\end{tikzpicture}
By Proposition \ref{prop_Nxv}, $\ON_i(a_1) \cap \ON_i(a_2) = \emptyset$ so case (1) can't happen.
Now consider case (2).
As $M_k(q)$ is vertex transitive, we can let $a_2 = [0,1]$, without loss of generality.
Then $a_1, a_t \in \N_0([0,1]) = \{[0, \omega^j] \mid j=1,2,\ldots, k-1\}$.
If we let $a_1 = [0, \omega^{j_1}]$ and $a_t = [0, \omega^{j_2}]$, for some $1 \leq j_1 \neq j_2 \leq k-1$, then $a_1 \stackrel{i}{\to} a_t$ implies $0 = \omega^{j_1} \cdot 0 - 0 \cdot \omega^{j_2} \in S_{k,i}$, which is a contradiction. Case (3) is isomorphic to case (2). 
So, if $a_1 \stackrel{0}{\longleftrightarrow} a_2$, then case (4) is the only possibility, which inductively implies that $T_l$ is monochromatic in color $0$.
Thus, by Proposition \ref{prop_Gamma} (1), $T_l$ must be contained in a color $0$ clique of $\Gamma_0$ and so $l \leq k \leq m$.

Now assume $a_1 \stackrel{i}{\to} a_2$ in $M_k(q)$.
If $l \geq 2$, consider $a_t$ for $3 \leq t \leq l$.
Again, we see that there are four possibilities for the triangle $(a_1, a_2, a_t)$ in $M_k(q)$:\\
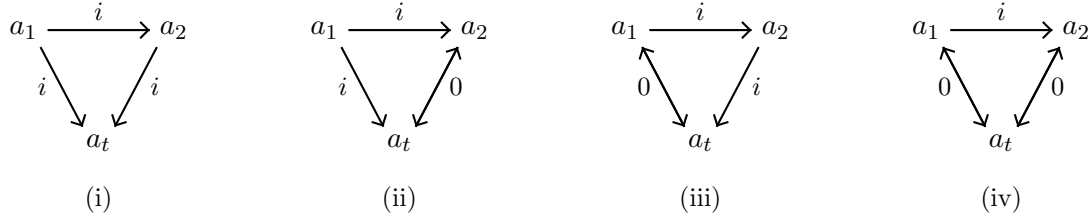
\begin{tikzpicture}
\node[inner sep=2pt, circle] (A) at (-2,0) {$a_1$};
\node[inner sep=2pt, circle] (B) at (0,0) {$a_2$}; 
\node[inner sep=2pt, circle] (C) at (-1,-1.5) {$a_t$}; 

\draw[thick, mid arrow] (A.east) to (B.west);
\draw[thick, mid arrow] (A.south east) to (C.north west);
\draw[thick, mid arrow] (B.south west) to (C.north east);

\node at (-1,0.25) {\small{$i$}};
\node at (-1.75,-0.75) {\small{$i$}};
\node at (-0.25,-0.75) {\small{$i$}};
\node at (-1.0,-2.25) {\small{(i)}};


\node[inner sep=2pt, circle] (A) at (2,0) {$a_1$};
\node[inner sep=2pt, circle] (B) at (4,0) {$a_2$}; 
\node[inner sep=2pt, circle] (C) at (3,-1.5) {$a_t$}; 

\draw[thick, mid arrow] (A.east) to (B.west);
\draw[thick, mid arrow] (A.south east) to (C.north west);
\draw[thick, mid arrow] (B.south west) to (C.north east);
\draw[thick, mid arrow] (C.north east) to (B.south west);

\node at (3,0.25) {\small{$i$}};
\node at (2.25,-0.75) {\small{$i$}};
\node at (3.75,-0.75) {\small{$0$}};
\node at (3,-2.25) {\small{(ii)}};


\node[inner sep=2pt, circle] (A) at (6,0) {$a_1$};
\node[inner sep=2pt, circle] (B) at (8,0) {$a_2$}; 
\node[inner sep=2pt, circle] (C) at (7,-1.5) {$a_t$}; 

\draw[thick, mid arrow] (A.east) to (B.west);
\draw[thick, mid arrow] (A.south east) to (C.north west);
\draw[thick, mid arrow] (C.north west) to (A.south east);
\draw[thick, mid arrow] (B.south west) to (C.north east);

\node at (7,0.25) {\small{$i$}};
\node at (6.25,-0.75) {\small{$0$}};
\node at (7.75,-0.75) {\small{$i$}};
\node at (7,-2.25) {\small{(iii)}};


\node[inner sep=2pt, circle] (A) at (10,0) {$a_1$};
\node[inner sep=2pt, circle] (B) at (12,0) {$a_2$}; 
\node[inner sep=2pt, circle] (C) at (11,-1.5) {$a_t$}; 

\draw[thick, mid arrow] (A.east) to (B.west);
\draw[thick, mid arrow] (A.south east) to (C.north west);
\draw[thick, mid arrow] (C.north east) to (B.south west);
\draw[thick, mid arrow] (C.north west) to (A.south east);
\draw[thick, mid arrow] (B.south west) to (C.north east);

\node at (11,0.25) {\small{$i$}};
\node at (10.25,-0.75) {\small{$0$}};
\node at (11.75,-0.75) {\small{$0$}};
\node at (11,-2.25) {\small{(iv)}};
\end{tikzpicture}

\noindent
Case (ii) can't happen because $\IN_i(a_2) \cap \IN_i(a_t) = \emptyset$, by Proposition \ref{prop_Nxv}.
Case (iv) is isomorphic to case (2) above, which we've seen is not possible.
We now examine case (iii).
As $M_k(q)$ is vertex transitive, we can let $a_1 = [0,1]$, without loss of generality.
Then $a_2 \in \ON_i([0, 1]) = \{ [ \omega^{i-1}, d] \mid d \in \mathbb{F}_q \}$
and $a_t \in \N_0([0,1]) = \{[0, \omega^j] \mid j=1,2,\ldots, k-1\}$.
Further, 
\begin{align*}
a_2 \stackrel{i}{\to} a_t 
& \Longleftrightarrow
[\omega^{i-1}, d] \stackrel{i}{\to} [0, \omega^j] 
\\ & \Longleftrightarrow
d\cdot 0 - \omega^{i-1}\cdot \omega^j \in S_{k,i}
\\ & \Longleftrightarrow
\omega^{i+j-1} \in -S_{k,i} = \{\omega^{kv+i-1+\frac{k}{2}} \mid v=0,1,\ldots, \tfrac{q-1}{k}-1\}
\\ & \Longleftrightarrow
\omega^{j} \in \{\omega^{kv+\frac{k}{2}} \mid v=0,1,\ldots, \tfrac{q-1}{k}-1\}
\\ & \Longleftrightarrow
j=\tfrac{k}{2}
\\ & \Longleftrightarrow
a_t=[0,\omega^{\frac{k}{2}}] = [0,-1]
\end{align*}
So, case (iii) is possible but there is only one possible $a_t$, which means there is only one value of $t \in \{3, \ldots, l\}$ for which $a_1 \stackrel{0}{\longleftrightarrow}  a_t$. 
So assume there is an $s \in \{3, \ldots, l\}$ such that\\
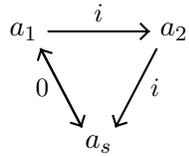
\begin{tikzpicture}
\node[inner sep=2pt, circle] (A) at (-2,0) {$a_1$};
\node[inner sep=2pt, circle] (B) at (0,0) {$a_2$}; 
\node[inner sep=2pt, circle] (C) at (-1,-1.5) {$a_s$}; 

\draw[thick, mid arrow] (A.east) to (B.west);
\draw[thick, mid arrow] (A.south east) to (C.north west);
\draw[thick, mid arrow] (C.north west) to (A.south east);
\draw[thick, mid arrow] (B.south west) to (C.north east);

\node at (-1,0.25) {\small{$i$}};
\node at (-1.75,-0.75) {\small{$0$}};
\node at (-0.25,-0.75) {\small{$i$}};
\end{tikzpicture}

\noindent
Then $a_1 \stackrel{i}{\to} a_t$ for all $t \in \{3, \ldots, l\} \setminus \{s\}$ and by previous arguments we must have\\
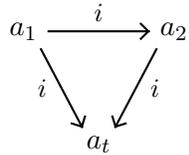
\begin{tikzpicture}
\node[inner sep=2pt, circle] (A) at (-2,0) {$a_1$};
\node[inner sep=2pt, circle] (B) at (0,0) {$a_2$}; 
\node[inner sep=2pt, circle] (C) at (-1,-1.5) {$a_t$}; 

\draw[thick, mid arrow] (A.east) to (B.west);
\draw[thick, mid arrow] (A.south east) to (C.north west);
\draw[thick, mid arrow] (B.south west) to (C.north east);

\node at (-1,0.25) {\small{$i$}};
\node at (-1.75,-0.75) {\small{$i$}};
\node at (-0.25,-0.75) {\small{$i$}};
\end{tikzpicture}

\noindent
Therefore, if $t_1, t_2 \in \{3, \ldots, l\} \setminus \{s\}$ with $t_1 < t_2$, then\\
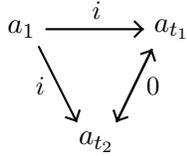
\begin{tikzpicture}
\node[inner sep=2pt, circle] (A) at (-2,0) {$a_1$};
\node[inner sep=2pt, circle] (B) at (0,0) {$a_{t_1}$}; 
\node[inner sep=2pt, circle] (C) at (-1,-1.5) {$a_{t_2}$}; 

\draw[thick, mid arrow] (A.east) to (B.west);
\draw[thick, mid arrow] (A.south east) to (C.north west);
\draw[thick, mid arrow] (B.south west) to (C.north east);
\draw[thick, mid arrow] (C.north east) to (B.south west);

\node at (-1,0.25) {\small{$i$}};
\node at (-1.75,-0.75) {\small{$i$}};
\node at (-0.25,-0.75) {\small{$0$}};
\end{tikzpicture}

\noindent
is not possible, by Proposition \ref{prop_Nxv}, and so $a_{t_1} \stackrel{i}{\to} a_{t_2}$.
Thus, if we remove $a_s$ from $T_l$ we get a monochromatic, in color $i$,  transitive subtournament of $M_k(q)$ of order $l-1$, which we call $T_{l-1}$.
Furthermore, $T_{l-1} \setminus \{a_1\}$ is a monochromatic, in color $i$,  transitive subtournament of $M_k(q)$ of order $l-2$.
If we let $H$ denote the induced subgraph of $M_k(q)$ with vertex set $\ON_i([0,1])$, then by Proposition \ref{prop_ONPaley}, $T_{l-1} \setminus \{a_1\} \subseteq H \cong P_k(q)$.
So, if $P_k(q)$ contains no monochromatic transitive subtournament of order $m$, then $l-2 < m$.

If there is no $3 \leq t \leq l$ for which $(a_1, a_2, a_t)$ satisfies cases (ii), (ii) or (iv) then all $a_t$, for $3 \leq t \leq l$, satisfy case (i).
Then $a_{t_1} \stackrel{i}{\to} a_{t_2}$ for all $3 \leq t_1 < t_2 \leq l$ by previous arguments.
So, in this case, $T_l$ itself is a monochromatic, in color $i$, transitive subtournament of $M_k(q)$.
Letting $H$ denote the induced subgraph of $M_k(q)$ with vertex set $\ON_i(a_1)$ and, again, using Proposition \ref{prop_ONPaley}, we get that $T_{l} \setminus \{a_1\} \subseteq H \cong P_k(q)$.
So, if $P_k(q)$ contains no monochromatic transitive subtournament of order $m$, then $l-1 < m$.

Overall, if $P_k(q)$ contains no monochromatic transitive subtournament of order $m$, then $M_k^{\ast}(q)$ contains no monochromatic transitive subtournament of order $m+2$.
\end{proof}

\begin{cor}\label{cor_Main}
Let $k \geq 2$ be an even integer and $q$ be a prime power such that $q \equiv k+1 \imod {2k}$.
If $\mathcal{K}_m(G_k(q))=0$, for $m\geq k$, then $R_{\frac{k}{2}}(m+2) \geq k(q+1)+1$.
\end{cor}

\begin{proof}
By definition, $\mathcal{K}_m(G_k(q))=0$ means that $G_k(q)$ contains no transitive subtournaments of order $m$. By the discussion at the start of this section, this implies $P_k(q)$ contains no transitive subtournaments of order $m$ \cite{MS}. Consequently, by Theorem \ref{thm_Main}, $M_k^{\ast}(q)$ contains no monochromatic transitive subtournament of order $m+2$. 
Recall, $M_k^{\ast}(q)$ is a tournament of order $n=k(q+1)$ whose edges are colored in $\frac{k}{2}$ colors, so $R_{\frac{k}{2}}(m+2) \geq k(q+1)+1$.
\end{proof}


\section{Application of Corollary \ref{cor_Main}}\label{sec_Results}
We now examine properties of $G_k(q)$ and apply Corollary \ref{cor_Main} to get improved lower bounds for certain directed Ramsey numbers.

We start with the case when $k=2$.
For all appropriate $q \leq 1583$ we found, by computer search, the order of the largest transitive subtournament of $G_2(q)$.
Then, from this data, we identified the largest $q$ such that $\mathcal{K}_m(G_k(q))=0$, for each $3 \leq m \leq 20$.
Call this $q_m$.
We then apply Corollary \ref{cor_Main} which yields $R(m+2) \geq \max(2(q_m+1)+1,q_{m+2}+1)$.
The results for $7 \leq m \leq 20$ are shown in Table \ref{tab_Largest_q_k2}.
($R(m)$ for $3 \leq m \leq 6$ are already known.)

\begin{table}[!htbp]
\centering
\begin{tabular}{| c | c | c | c | c | c | c | c | c | c | c | c | c | c | c |}
\hline
$m$ & $7$ & $8$ & $9$ & $10$ & $11$ & $12$ & $13$ & $14$ & $15$ & $16$ & $17$ & $18$ & $19$ & $20$\\
\hline
\hline
$q_m$  & $27$ & $47$ & $83$ & $107$ & $107$ & $199$ & $271$ & $367$ & $443$ & $619$ & $659$ & $971$ & $1259$ & $1571$\\
\hline
$R(m) \geq$ & $28$ & \textbf{57} & \textit{84} & \textit{108} & \textbf{169} & \textbf{217} & \textit{272} & \textbf{401} & \textbf{545} & \textbf{737} & \textbf{889} & \textbf{1241} & \textbf{1321} & \textbf{1945}\\
\hline
\end{tabular}
\vspace{3pt}
\caption{Lower Bounds for $R(m)$.}
\label{tab_Largest_q_k2}
\end{table}

The values of $q_m$ in Table \ref{tab_Largest_q_k2}, for $7 \leq m \leq 18$, confirm those of Sanchez-Flores \cite{SF2}, and, for $m=19$, that of Exoo \cite{ES}.
The best known lower bound for $m=7$ is $R(7) \geq 34$, due to Neiman, Mackey and Heule \cite{NMH}.
For $8 \leq m \leq 10$ and $12 \leq m \leq 19$ the previously best known lower bound was $R(m) \geq q_m+1$ \cite{ES}.
Also from \cite{ES} we have that $R(11) \geq 112$. 
So the values in bold in Table \ref{tab_Largest_q_k2}  represent an improvement to the previously best known lower bounds and the values in italics equal the best known lower bounds.

We also performed a similar exercise for $k =4,6,8$ and $10$, identifying, in each case, the largest $q$ such that $\mathcal{K}_m(G_k(q))=0$, for $3 \leq m \leq 10$. 
We will denote such $q$ as $q_{m,k}$.
Table \ref{tab_Largest_q_k>2} outlines these values.
The values in the last row of the table indicate the upper limit for $q$ in our search.
Note that values of $q_{m,k}$ close to this limit will not be optimal.

\begin{table}[!htbp]
\centering
\begin{tabular}{| c | c | c | c | c |}
\hline
$m$ & $k=4$ & $k=6$ & $k=8$ & $k=10$ \\
\hline
$3$ & $13$ & $43$ & $169$ & $71$ \\
$4$ & $125$ & $343$ & $953$ & $3331$ \\
$5$ & $157$ & $859$ & $2809$ & $6791$ \\
$6$ & $829$ & $4339$ & $15625$ & $33191$ \\
$7$ & $709$ & $4423$ & $26153$ & $43411$ \\
$8$ & $1709$ & $18523$ & $29929$ & $58771$ \\
$9$ & $3517$ & $29611$ & $29929$ & $59951$ \\
$10$ & $7573$ & $29959$ & $29929$ & $59971$ \\
\hline
$q<$ & $10000$ & $30000$ & $30000$ & $60000$\\
\hline
\end{tabular}
\vspace{6pt}
\caption{Largest $q$ found such that $\mathcal{K}_m(G_k(q))=0$.}
\label{tab_Largest_q_k>2}
\end{table}
\vspace{-12pt}
Now, $R_{\frac{k}{2}}(m) \geq q_{m,k}+1$, and, by Corollary \ref{cor_Main}, $R_{\frac{k}{2}}(m+2) \geq k(q_{m,k}+1)+1$ when $m \geq k$.
We note also that for $t \geq 2$ \cite[Prop. 5]{MT}
\begin{equation*}\label{for_LBMult}
R_t(m) \geq (R_{t-1}(m)-1)(R(m)-1) +1.  
\end{equation*}
It is already known that 
$R(3)=4$, $R(4)=8$ \cite{EM}, $R(5)=14$ \cite{RP}, $R(6)=28$ \cite{SF1},  $R(7) \geq 34$ \cite{NMH}, $R_2(3)=14$ \cite{BD}, $R_2(4) \geq 126$ and $R_3(3) \geq 44$ \cite{MS}.
We combine all this information, including values from Table \ref{tab_Largest_q_k2}, to get lower bounds on the Ramsey numbers $R_t(m)$ for $t \geq 2$ and $3 \leq m \leq 10$.
The results are shown in Table \ref{tab_LB_k>2}.

\begin{table}[!htbp]
\centering
\begin{tabular}{| c | c | c | c | c |}
\hline
$m$ & $t=2$ & $t=3$ & $t=4$ & $t \geq 5$ \\
\hline
\hline
$3$ & $14$ & $44$ & $170$ & $169 \cdot 3^{t-4}+1$ \\
\hline
$4$ & $126$ & \multicolumn{3}{| c |}{$125 \cdot 7^{t-2} +1$}\\
\hline
$5$ & \multicolumn{4}{| c |}{$13^{t} +1$}  \\
\hline
$6$ & $830$ & \multicolumn{3}{| c |}{$829 \cdot 27^{t-2} +1$}\\
\hline
$7$ & \multicolumn{4}{| c |}{$33^{t} +1$}  \\
\hline
$8$ & $3321$ & \multicolumn{3}{| c |}{$3320 \cdot 56^{t-2} +1$}\\
\hline
$9$ & \multicolumn{4}{| c |}{$83^{t} +1$}  \\
\hline
$10$ & \multicolumn{4}{| c |}{$107^{t} +1$}  \\
\hline
\end{tabular}
\vspace{6pt}
\caption{Lower bounds for $R_t(m)$.}
\label{tab_LB_k>2}
\end{table}
\vspace{-15pt}
\noindent
The general formulas in the cases $m=3,6, 8$ improve on what was previously known.
We note that the $m=8$ case is the only one where Corollary \ref{cor_Main} influences the results.


\end{document}